\documentclass[a4paper]{amsart}

\usepackage{amsmath,amsthm, amssymb}
\usepackage{amscd}
\usepackage{amsfonts}
\usepackage{bm}
\usepackage{listings}
\usepackage[all]{xy}

\newcounter{codecount}
\lstnewenvironment{code}[2]
{

\setcounter{lstlisting}{\thecodecount}
\refstepcounter{codecount}
\lstset{%
label={#1},
caption={#2},
basicstyle=\ttfamily\small,
frame=tb
}
}{}

\newcounter{resultcount}
\lstnewenvironment{result}[2]
{

\setcounter{lstlisting}{\theresultcount}
\refstepcounter{resultcount}
\lstset{%
label={#1},
caption={#2},
basicstyle=\ttfamily\footnotesize,
frame=tb,
tabsize=2,
breaklines=ture,
breakindent=5pt
}
}{}

\theoremstyle{definition}
\newtheorem{theorem}{Theorem}[section]
\newtheorem{prop}[theorem]{Proposition}
\newtheorem{lemma}[theorem]{Lemma}
\newtheorem{cor}[theorem]{Corollary}

\newtheorem{remark}[theorem]{Remark}

\newtheorem{note}[theorem]{Notation}

\numberwithin{theorem}{section}

\newcommand\SL{\operatorname{SL}}

\newcommand\PGL{\operatorname{PGL}}
\newcommand\PSL{\operatorname{PSL}}

\newcommand\Sing{\operatorname{Sing}}
\newcommand\Stab{\operatorname{Stab}}
\newcommand\Gal{\operatorname{Gal}}

\author{Yusuke Yoshida}

\title{$\mathfrak{A}_{6}$-invariant curves of genera $10$ and $19$}

\begin{document}
\maketitle

\begin{abstract}
We study smooth curves on which the alternating group $\mathfrak{A}_{6}$ acts faithfully. Let $\mathcal{V} \subset \PGL(3, \mathbb{C})$ be the Valentiner group, which is isomorphic to $\mathfrak{A}_{6}$. We see that there are integral $\mathcal{V}$-invariant curves of degree $12$ which have geometric genera $10$ and $19$. On the other hand, if $\mathfrak{A}_{6}$ acts faithfully on a curve $C$ of genus $10$ or $19$, then we give an explicit description of the extension $k(C / \mathfrak{A}_{5}) \rightarrow k(C / \mathfrak{A}_{6})$ for any icosahedral subgroup $\mathfrak{A}_{5}$. Using this, we show the uniqueness of smooth projective curves of genera $10$ and $19$ whose automorphism groups contain $\mathfrak{A}_{6}$.
\end{abstract}

\section{Introduction}

Automorphism groups of algebraic curves have long been studied. Recently, Harui gave a classification of automorphism groups of smooth plane curves over $\mathbb{C}$ (\cite{Harui}). In this classification, ``primitive'' subgroups of $\PGL(3, \mathbb{C})$ occupies an important part. They are either conjugate to the Valentiner group $\mathcal{V} \cong \mathfrak{A}_{6}$, the icosahedral group $\mathcal{I} \cong \mathfrak{A}_{5}$, the Klein group $\mathcal{K} \cong \PSL(2, \mathbb{F}_{7})$, the Hessian group $H_{216}$ or its subgroup $H_{72}$ or $H_{36}$. In \cite{Y}, the author studied projective plane curves invariant under $G = \mathcal{V}$, $\mathcal{I}$ or $\mathcal{K}$, and determined all degrees of nonsingular (resp. integral) curves whose automorphism groups contain $G$. For the Valentiner group $\mathcal{V}$, the result is as follows.
\begin{theorem}\label{deg of V inv cur}
({\cite[Theorem 3.7, Theorem 4.4]{Y}})
Let $d$ be a positive integer.
\begin{enumerate}
\item There is a nonsingular projective plane curve of degree $d$ whose automorphism group is $\mathcal{V}$ (on equivalently contains $\mathcal{V}$) if and only if $d \equiv 0$, $6$ or $12 \mod 30$.
\item There exists an integral projective plane curve of degree $d$ invariant under $\mathcal{V}$ if and only if $d$ is a multiple of $6$, $d \neq 18$ and $d \neq 24$.
\end{enumerate}
\end{theorem}

The next step would be to study, for each degree, the existence of curves invariant under such a group of a given\emph{ geometric genus}. This is also related to the problem of finding birational plane models of $\mathfrak{A}_{6}$-invariant curves. For example, by Theorem \ref{deg of V inv cur}, a smooth plane model of an $\mathfrak{A}_{6}$-invariant curve of genus $g = \dfrac{1}{2}(d - 1)(d - 2)$ for $d \equiv 0$, $6$ or $12 \mod 30$ is given.

A smooth projective curve over $\mathbb{C}$ can be also seen as a compact Riemann surface. Automorphism groups of compact Riemann surfaces have been studied with the help of the Fuchsian groups. According to Breuer's book \cite[Ch.5, \S19 -- \S20]{Breuer}, the genera of compact Riemann surfaces whose automorphism groups have order $360$ are $10$, $19$ and so on. We have a genus $10$ curve whose automorphism group is isomorphic to $\mathfrak{A}_{6}$ as a $\mathcal{V}$-invariant smooth plane curve of degree $6$, but there is no smooth projective plane curve of genus $19$. On the other hand, if we take singular curves into consideration, then it turns out that we can give a plane model invariant under $\mathcal{V}$. In this paper, we find $\mathcal{V}$-invariant projective plane curves of degree $12$ with genera $10$ and $19$ and show the birational uniqueness in each genus:
\begin{theorem}\label{main theorem}
Let $C$ be a smooth projective curve of genus $g$ on which the alternating group $\mathfrak{A}_{6}$ acts faithfully. For $g = 10$ and $19$, $C$ is unique up to birational equivalence.
\begin{enumerate}
\item If $g = 10$, then $C$ is isomorphic to the following:
\begin{itemize}
\item the smooth projective plane curve of degree $6$ invariant under $\mathcal{V}$,
\item the normalization of the unique integral $\mathcal{V}$-invariant projective plane curve of degree $12$ with $45$ nodes.
\end{itemize}
\item If $g = 19$, then $C$ is isomorphic to the normalization of the unique integral $\mathcal{V}$-invariant projective plane curve of degree $12$ with $36$ nodes.
\end{enumerate}
\end{theorem}

\begin{remark}
For each case of Theorem \ref{main theorem}, we will give the specific form of $C$ in Section 2. We also give an explicit description of maps $\mathbb{P}^1 \rightarrow \mathbb{P}^{1}$ of degree $6$ whose Galois closure is the curve $C$.
\end{remark}

To find such curves, we use the description of the $\widetilde{\mathcal{V}}$-invariant ring $\mathbb{C}[x, y, z]^{\widetilde{\mathcal{V}}}$ where $\widetilde{\mathcal{V}}$ is a lift of $\mathcal{V}$ in $\SL(3, \mathbb{C})$. It is classically known that this ring $\mathbb{C}[x, y, z]^{\widetilde{\mathcal{V}}}$ is generated by explicitly given polynomials of degrees $6$, $12$, $30$ and $45$. The first polynomial defines a nonsingular curve of genus $10$ invariant under $\mathcal{V}$. Let $\mathfrak{d}_{12}$ be the linear system of the $\mathcal{V}$-invariant curves of degree $12$. We see that $\dim \mathfrak{d}_{12} = 1$, and its general elements are nonsingular by Bertini's theorem. Thus $\mathfrak{d}_{12}$ has finitely many singular members. By calculations with the computer algebra system SINGULAR (\cite{DGPS}), we find $5$ singular members and their singular points concretely. We can check that two of them are integral and have geometric genera $10$ and $19$. 

Next, we show the birational uniqueness for each of genera $g = 10$ and $19$. In each case, let $C$ be a curve of genus $g$ on which $\mathfrak{A}_{6}$ acts faithfully. Since the alternating group $\mathfrak{A}_{5}$ can be considered as a subgroup of $\mathfrak{A}_{6}$, we have the natural morphisms $\pi_{\mathfrak{A}_{6}} : C \rightarrow C / \mathfrak{A}_{6}$, $\pi_{\mathfrak{A}_{5}} : C \rightarrow C / \mathfrak{A}_{5}$ and $f : C / \mathfrak{A}_{5} \rightarrow C / \mathfrak{A}_{6}$ with $f \circ \pi_{\mathfrak{A}_{5}} = \pi_{\mathfrak{A}_{6}}$. Using Riemann-Hurwitz theorem, we see that the quotient curves $C / \mathfrak{A}_{6}$ and $C / \mathfrak{A}_{5}$ are isomorphic to $\mathbb{P}^{1}$. Hence, $f$ can be considered as a rational function. If we consider the corresponding field extension $f^{\ast} : K \rightarrow M$, then $C$ corresponding the Galois closure $L$. By calculations using information on the ramification of $f$, we find a rational function $f$, unique up to projective equivalence, with $\Gal(L / K) \cong \mathfrak{A}_{6}$.

The organization of this paper is as follows. In Section 2, we recall the description of the Valentiner group $\mathcal{V}$ and its invariant curves of degree $12$. After some calculations with SINGULAR, we find the $\mathcal{V}$-invariant (projective plane) curves of degree $12$ with genera $10$ and $19$. In Section 3, we consider a curve on which $\mathfrak{A}_{6}$ acts faithfully, and study the quotients $C / \mathfrak{A}_{6}$ and $C / \mathfrak{A}_{5}$. Using this, we show the uniqueness for the genus $10$ curve invariant under $\mathfrak{A}_{6}$. In Section 4, we show the uniqueness for the genus $19$ using similar arguments.

\section{$\mathcal{V}$-invariant curves of degree 12}

There is a subgroup $\mathcal{V}$ of $\PGL(3, \mathbb{C})$, called the \textit{Valentiner group}, which is isomorphic to the alternating group $\mathfrak{A}_{6}$. Up to conjugacy, $\mathcal{V}$ is generated by the equivalence classes of
\[
\left(
\begin{array}{ccc}
-1 & 0 & 0 \\
0 & 1 & 0 \\
0 & 0 & -1
\end{array}
\right),{\rm {\ }}
\left(
\begin{array}{ccc}
0 & 0 & 1 \\
1 & 0 & 0 \\
0 & 1 & 0
\end{array}
\right),{\rm {\ }}
\left(
\begin{array}{ccc}
1 & 0 & 0 \\
0 & 0 & \rho^{2} \\
0 & -\rho & 0
\end{array}
\right) {\rm and {\ }}
\frac{1}{2}
\left(
\begin{array}{ccc}
1 & \tau^{-1} & -\tau \\
\tau^{-1} & \tau & 1 \\
\tau & -1 & \tau^{-1}
\end{array}
\right)
\]
where $\rho = e^{\frac{2}{3} \pi i}$ and $\tau = \displaystyle{\frac{1+\sqrt{5}}{2}}$.

Take the preimage $\widetilde{\mathcal{V}}$ of $\mathcal{V}$ by the natural projection $\SL(3, \mathbb{C}) \rightarrow \PGL(3, \mathbb{C})$.  By {\cite[Lemma 2.4]{Y}}, any $\mathcal{V}$-invariant projective plane curve is define by a $\widetilde{\mathcal{V}}$-invariant homogeneous polynomial. In a suitable coordinate system, different from the one employed above, the polynomial
\[
F(x, y, z) := 10 x^{3} y^{3} + 9 x^{5} z + 9 y^{5} z - 45 x^{2} y^{2} z^{2} -135 xyz^{4} + 27 z^{6}
\]
of degree $6$ is invariant under $\widetilde{\mathcal{V}}$. The curve defined by $F(x, y, z)$ is called a \textit{Wiman's curve}(\cite{Wiman}). (Wiman studied another series of curves of degree $6$ whose automorphism groups are the symmetric group $\mathfrak{S}_{5}$. These curves are often called ``Wiman curves''. We remark that they are \emph{not} isomorphic to $V(F)$.) We define the homogeneous polynomial $\Phi(x, y, z)$ of degree $12$ as the Hessian of $F$, i.e.,
\[
\Phi(x, y, z) := \det H(F)(x, y, z) = \left|\begin{array}{ccc}
\dfrac{\partial^{2} F}{\partial x^{2}} & \dfrac{\partial^{2} F}{\partial x \partial y} & \dfrac{\partial^{2} F}{\partial x \partial z} \\
&&\\
\dfrac{\partial^{2} F}{\partial y \partial x} & \dfrac{\partial^{2} F}{\partial y^{2}} & \dfrac{\partial^{2} F}{\partial y \partial z} \\
&&\\
\dfrac{\partial^{2} F}{\partial z \partial x} & \dfrac{\partial^{2} F}{\partial z \partial y} & \dfrac{\partial^{2} F}{\partial z^{2}}
\end{array}\right|.
\]
Finally, the homogeneous polynomial $\Psi(x, y, z)$ of degree $30$ is defined as the ``border Hessian'' of $(F, \Phi)$, i.e.,
\[
\Psi(x, y, z) := \left|\begin{array}{ccc|c}
&&& \dfrac{\partial \Phi}{\partial x} \\
&&&\\
& H(F) & & \dfrac{\partial \Phi}{\partial y} \\
&&&\\
&&& \dfrac{\partial \Phi}{\partial z} \\
&&&\\
\hline
&&&\\
\dfrac{\partial \Phi}{\partial x} & \dfrac{\partial \Phi}{\partial y} & \dfrac{\partial \Phi}{\partial z} & 0
\end{array}\right|.
\]
Let $\mathfrak{d}_{d}$ be the linear system generated by $\mathcal{V}$-invariant curves of degree $d$. If $d$ is even, then any member of $\mathfrak{d}_{d}$ is defined by a linear combination of $F^{i} \Phi^{j} \Psi^{k}$ with $6i + 12j + 30k = d$. By \cite[Theorem 3.7]{Y}, a general element of $\mathfrak{d}_{d}$ is nonsingular for $d \equiv 0, 6$ or $12 \mod 30$.

We consider the lowest degree curve, i.e., the case of $d = 6$. The only $\mathcal{V}$-invariant curve of degree $6$ is $C = V(F)$. Since $C$ is nonsingular, it is a curve of genus $10$ invariant under $\mathcal{V}$.

Take $d = 12$. Then any element of $\mathfrak{d}_{12}$ is defined by a homogeneous polynomial $a F(x, y, z)^{2} + b \Phi(x, y, z)$ for a point $(a : b) \in \mathbb{P}^{1}$, and $\dim \mathfrak{d}_{12} = 1$. It is known that $V(\Phi)$ is a nonsingular curve. Thus the number of singular elements is finite. Let $P = \dfrac{a}{b}$ and $C_{P}$ the curve corresponding to the point $(a : b)$. Now, we will see that there are $5$ special points $P$ such that $C_{P}$ is singular.

\begin{prop}\label{five singular curves}
The curve $C_{P}$ is singular if and only if $P = \infty$, $20250$, $-10125$, $\xi$ or $\overline{\xi}$ where $\xi$ and $\overline{\xi}$ are solutions to the quadratic equation
\[
p^{2} + 3375p + 7593750 = 0.
\]
Except for $P = \infty$, the curve $C_{P}$ is reduced.
\end{prop}

\begin{proof}
By definition, the curve $C_{\infty}$ is defined by $F(x, y, z)^{2}$. Thus it is nonreduced and singular. Assume that $b \neq 0$. We look for all singular points of the curves $C_{P}$ defined by $P F(x, y, z)^{2} + \Phi(x, y, z)$ for all $P \in \mathbb{C}$.

We calculate the values of $P$ such that $C_{P}$ is singular using the computer algebra system SINGULAR (\cite{DGPS}). In the following, \texttt{F} denotes $F$ and \texttt{Phi} denotes the polynomial $\Phi$, i.e., the determinant of the Hessian matrix of \texttt{F}. We let \texttt{G} be the homogeneous polynomial $P F^{2} + \Phi$.
\begin{code}{setting}{Setting}
LIB "primdec.lib";
ring R = 0,(x,y,z),dp;

poly F = 10x3y3+9x5z+9y5z-45x2y2z2-135xyz4+27z6;

matrix HF = jacob(jacob(F));
//HF is the Hessian matrix of F.
poly Phi = det(HF);

ring S = 0,(x,y,z,P),dp;
map f = R,x,y,z;
poly G = P*(f(F)^2) + f(Phi);
\end{code}

To find singularities in that $z = 0$, let \texttt{J} be the ideal generated by \texttt{G} and its derivative of \texttt{G} with respect to $x$, $y$ and $z$. We put \texttt{jG0} the ideal \texttt{J} in $z = 0$, and check the radical ideal of \texttt{jG0}.
\begin{code}{}{}
ideal J = G,diff(G,x),diff(G,y),diff(G,z);
ideal jG0 = subst(J,z,0);
radical(jG0);
\end{code}
This code gives $3$ generators of the radical of \texttt{jG0}, and we obtain a system of equations
\[
\left\{
\begin{array}{l}
y(P + 10125) = 0, \\
x(P + 10125) = 0, \\
x^{5} + y^{5} = 0.
\end{array}
\right.
\]
Since $(x, y, z) \neq (0, 0, 0)$, we have $x \neq 0$ or $y \neq 0$, and $P = -10125$. Then $5$ points $(x : y : 0)$ with $x^5 + y^5 = 0$ satisfy the system of equations.

We consider the affine part $z \neq 0$, i.e., we look at the singular points on $V(P F^2 + \Phi) \subset \mathbb{C}^{3}$. Take the map $\texttt{g} : (x, y, z, P) \mapsto (x, y, 1, P)$. Then \texttt{g(G)} defines $V(P F^2 + \Phi)$. Let \texttt{jG1} be the ideal generated by \texttt{g(G)} and its derivatives with respect to $x$ and $y$. We eliminate $x$ and $y$ from \texttt{jG1}.
\begin{code}{}{}
ring T = 0,(x,y,P),dp;
map g = S,x,y,1,P;

ideal jG1 = g(G),diff(g(G),x),diff(g(G),y);
ideal J = eliminate(radical(jG1),xy);
J;
factorize(J[1]);
\end{code}
This code returns the polynomial
\[
P^4-6750P^3-231609375P^2-768867187500P-1556956054687500
\]
in $P$, which can be factored into
\[
(P - 20250)(P + 10125)(P^{2} + 3375P + 7593750).
\]
Furthermore, the following code returns the colength of the radical ideal \texttt{jG1}.
\begin{code}{}{}
vdim(std(radical(jG1)));
\end{code}
This returns $196$. Hence, for $z \neq 0$, the curve $C_{P}$ has singularities only at finitely many points.

Therefore, if $C_{P}$ has a singular point, then $P = 20250$, $-10125$ or a solution of $p^{2} + 3375p + 7593750 = 0$. Since the number of singular points of $C_{P}$ is finite, $C_{P}$ is reduced.
\end{proof}

We will show that $C_{-10125}$ is a $\mathcal{V}$-invariant curve of genus $10$ and $C_{20250}$ is of genus $19$. First, we consider irreducible components of $C_{P}$.

We see that the alternating group $\mathfrak{A}_{6}$ has the $22$ conjugacy classes of subgroups as in Table \ref{A6 subgp}.

\begin{note}
Let $\mathcal{C}_{n}$ be the cyclic group of order $n$ and $\mathcal{D}_{n}$ the dihedral group of order $2n$.
\end{note}

{\footnotesize
\begin{table}[htb]
\caption{The conjugacy classes of subgroups of $\mathfrak{A}_{6}$}
\begin{tabular}{|c|c|c|c|}
\hline
No. & isomorphic group & order & index \\
\hline
$1$ & $1$ & $1$ & $360$\\
\hline
$2$ & $\mathcal{C}_{2}$ & $2$ & $180$\\
\hline
$3$ & $\mathcal{C}_{3}$ & $3$ & $120$\\
\hline
$4$ & $\mathcal{C}_{3}$ & $3$ & $120$\\
\hline
$5$ & $\mathcal{C}_{2}^{2}$ & $4$ & $90$\\
\hline
$6$ & $\mathcal{C}_{2}^{2}$ & $4$ & $90$\\
\hline
$7$ & $\mathcal{C}_{4}$ & $4$ & $90$\\
\hline
$8$ & $\mathcal{C}_{5}$ & $5$ & $72$\\
\hline
$9$ & $\mathfrak{S}_{3}$ & $6$ & $60$\\
\hline
$10$ & $\mathfrak{S}_{3}$ & $6$ & $60$\\
\hline
$11$ & $\mathcal{D}_{4}$ & $8$ & $45$\\
\hline
$12$ & $\mathcal{C}_{3}^{2}$ & $9$ & $40$\\
\hline
$13$ & $\mathcal{D}_{5}$ & $10$ & $36$\\
\hline
$14$ & $\mathfrak{A}_{4}$ & $12$ & $30$\\
\hline
$15$ & $\mathfrak{A}_{4}$ & $12$ & $30$\\
\hline
$16$ & $\mathcal{C}_{3} \rtimes \mathfrak{S}_{3}$ & $18$ & $20$\\
\hline
$17$ & $\mathfrak{S}_{4}$ & $24$ & $15$\\
\hline
$18$ & $\mathfrak{S}_{4}$ & $24$ & $15$\\
\hline
$19$ & $\mathcal{C}_{3}^{2} \rtimes \mathcal{C}_{4}$ & $36$ & $10$\\
\hline
$20$ & $\mathfrak{A}_{5}$ & $60$ & $6$\\
\hline
$21$ & $\mathfrak{A}_{5}$ & $60$ & $6$\\
\hline
$22$ & $\mathfrak{A}_{6}$ & $360$ & $1$\\
\hline
\end{tabular}
\label{A6 subgp}
\end{table}
}

\begin{lemma}\label{reducible curves}
Assume that $C_{P}$ is reducible and write $C_{P} = \displaystyle\sum_{i = 1}^{n} D_{i}$ where $D_{1}, \cdots, D_{n}$ are integral curves in $\mathbb{P}^2$. Then one of the following holds:
\begin{enumerate}
\item[(i)] $n = 2$ and $D_{1} = D_{2} = V(F)$. Thus $P = \infty$ and $C_{P} = 2 V(F)$ is nonreduced.
\item[(ii)] $n = 6$ and $\deg D_{i} = 2$ for any $i$. Then $C_{P}$ is reduced and reducible.
\item[(iii)] $n = 12$ and $\deg D_{i} = 1$ for any $i$. Then $C_{P}$ is reduced and reducible.
\end{enumerate}
\end{lemma}

\begin{proof}
Let $\mathfrak{O}$ be the set $\{D_{1}, \cdots, D_{n}\}$ with $n \geq 2$. Since $C_{P}$ is invariant under $\mathcal{V}$, $\mathcal{V}$ acts on $\mathfrak{O}$. Then
\[
\frac{|\mathcal{V}|}{|\Stab(D_{i})|} = |\mathcal{V} \cdot D_{i}|
\]
for any $i$ where $\mathcal{V} \cdot D_{i}$ is the $\mathcal{V}$-orbit of $D_{i}$ in $\mathfrak{O}$. Since
\[
|\mathcal{V} \cdot D_{i}| \leq n \leq \sum_{i = 1}^{n} \deg D_{i} = \deg C_{P} = 12,
\]
we obtain $|\Stab(D_{i})| \geq 30$. If a subgroup $H$ of $\mathcal{V}$ satisfies $|H| \geq 30$, then $H$ is $\mathcal{V}$ or isomorphic to  $\mathfrak{A}_{5}$ or $\mathcal{C}_{3}^{2} \rtimes \mathcal{C}_{4}$ by Table \ref{A6 subgp}. Hence, any element $D_{i}$ is invariant under one of these groups.

Suppose that one of $D_{i}$ is invariant under $\mathcal{V}$. Since $\deg D_{i} < 12$, $\deg D_{i} = 6$ and $D_{i} = V(F)$. Thus $a F^{2} + b \Phi$ is divisible by $F$. Since $\Phi$ is not divisible by $F$, $C_{P}$ is defined by $a F^{2}$ with $a \neq 0$. Thus the statement (i) holds.

Suppose that $D_{i}$ is not $\mathcal{V}$-invariant for any $i$. Since $\Stab(D_{i}) \cong \mathfrak{A}_{5}$ or $\mathcal{C}_{3}^{2} \rtimes \mathcal{C}_{4}$, we have $|\mathcal{V} \cdot D_{i}| = 6$ or $10$ for each $i$. Thus we see
\[
\begin{array}{rcl}
\displaystyle\sum_{i = 1}^{n} \deg D_{i} & = & \displaystyle\sum_{\Stab(D_{i}) \cong C_{3}^{2} \rtimes C_{4}} \deg D_{i} + \displaystyle\sum_{\Stab(D_{i}) \cong \mathfrak{A_{5}}} \deg D_{i}\\
&&\\
& = & 10 k + 6 l
\end{array}
\]
where $k$ and $l$ are nonnegative integers. Since the left hand side is equal to $12$, the only possible pair is $(k, l) = (0, 2)$. Then $\Stab(D_{i}) = \mathfrak{A}_{5}$ for any $i$, and $n = 6$ or $12$. Hence, the statement (ii) or (iii) holds. Then $C_{P}$ is not equal to $C_{\infty}$. By Proposition \ref{five singular curves}, $C_{P}$ is reduced.
\end{proof}

\begin{remark}
We can show that the curves $C_{\xi}$ and $C_{\overline{\xi}}$ are reducible and are unions of six conics, and they are the only curves $C_{P}$ satisfying the condition (ii). Furthermore, there is no curve $C_{P}$ which satisfies the condition (iii).
\end{remark}

\begin{lemma}\label{orbit size}
Let $\mathcal{O} \subset \mathbb{P}^{2}$ be a $\mathcal{V}$-orbit. Then $|\mathcal{O}| \geq 30$.
\end{lemma}

\begin{proof}
We consider the stabilizer of a point. Let $H$ be a subgroup of $\mathcal{V}$ whose index is less then $30$ (i.e., groups No. 16--22 in Table \ref{A6 subgp}). Then $n := [\mathcal{V} : H] = 1$, $6$, $10$, $15$ or $20$.

Assume that $H$ fixes a point. By the representation theory of finite groups, $H$ also fixes a line in $\mathbb{P}^{2}$. (See Remark \ref{fixed line}.) Let $C'$ be a union of $n$ lines is invariant under $\mathcal{V}$. For $n = 1$, $C'$ is a line. Otherwise, $C'$ is a reducible or nonreduced curve of degree $n$ invariant under $\mathcal{V}$. Since the only $\mathcal{V}$-invariant curve of degree $6$ is $V(F)$ and it is irreducible, we see $C' \neq V(F)$, thus $n \neq 6$. On the other hand, the degree $d$ of a $\mathcal{V}$-invariant curve is divisible by $6$ or is greater than $45$. Hence, $n$ is neither $1$, $10$, $15$ nor $20$. Therefore, there is no $\mathcal{V}$-orbit whose size is less than $30$.
\end{proof}

\begin{remark}\label{fixed line}
There is a $H$-invariant subspace $V$ of dimension $1$ in $\mathbb{C}^{3}$. Since $H$ is finite, the (orthogonal) compliment subspace $L$ of $V$ is invariant under $H$ and dimension $2$. The subspace $L$ corresponds to a line.
\end{remark}

We look at the singularities of $C_{-10125}$ and $C_{20250}$.

\begin{prop}\label{irreducibility}
The curves $C_{-10125}$ and $C_{20250}$ satisfy the following.
\begin{enumerate}
\item The singular locus of $C_{20250}$ is a $\mathcal{V}$-orbit $\mathcal{O}_{36}$ of order $36$.
\item The singular locus of $C_{-10125}$ is a $\mathcal{V}$-orbit $\mathcal{O}_{45}$ of order $45$.
\item $C_{-10125}$ and $C_{20250}$ are nodal.
\item $C_{-10125}$ and $C_{20250}$ are irreducible.
\end{enumerate}
\end{prop}

\begin{proof}
We prove the claims (1), (2) and (3) by calculations with SINGULAR. We use Code \ref{setting} in the proof of Proposition \ref{five singular curves}. First, suppose that $P = 20250$. Then no singular point of $C_{20250}$ is contained in $V(z)$. Thus we may assume that $z = 1$. Let \texttt{f} be the ring homomorphism \texttt{S} $\rightarrow$ \texttt{T} defined as $z \mapsto 1$ and $P \mapsto 20250$ where \texttt{T} is the polynomial ring $\mathbb{Q}[x, y]$. Then the affine curve $C_{20250} - V(z)$ embedded in $\mathbb{A}^{2}$ is defined by \texttt{f(G)}. Take the ideal \texttt{jfG} generated by \texttt{f(G)} and its derivatives with respect to $x$ and $y$. We can calculate the dimension of the quotient ring of \texttt{T} modulo the ideal generated by the initial terms of the standard Gr\"{o}bner basis of \texttt{jfG} with Code \ref{setting} and the following code.
\begin{code}{}{}
ring T = 0,(x,y),dp;
map f = S,x,y,1,20250;

ideal jfG = f(G),diff(f(G),x),diff(f(G),y);
vdim(std(jfG));
\end{code}
The code returns $36$, i.e., the colength of \texttt{jfG} is $36$. Since the set $\Sing(C_{20250})$ is a $\mathcal{V}$-orbit, if the multiplicity of \texttt{jfG} was greater than $1$ at a singular point, then the number of the singular points is at most $18$. This contradicts Lemma \ref{orbit size}. Therefore, all singularities are nodes.

Next, suppose that $P = -10125$. By the proof of Proposition \ref{five singular curves}, $C_{-10125}$ has $5$ singular points on $V(z)$. Now, we check the singular points of $C_{-10125}$ on $\mathbb{P}^{2} - V(z)$.  Let \texttt{g} be the homomorphism \texttt{S} $\rightarrow$ \texttt{T} defined by $z \mapsto 1$ and $P \mapsto -10125$. Then the affine curve $C_{-10125} - V(z)$ embedded to $\mathbb{A}^{2}$ is defined by \texttt{g(G)}. Take the ideal \texttt{jgG} generated by \texttt{g(G)} and its derivatives with respect to $x$ and $y$. We can calculate the dimension of the quotient ring of \texttt{T} modulo the ideal generated by the initial terms of the standard Gr\"{o}bner basis of \texttt{jgG} with the following code.
\begin{code}{}{}
ring T = 0,(x,y),dp;
map g = S,x,y,1,-10125;

ideal jgG =  g(G),diff(g(G),x),diff(g(G),y);
vdim(std(jgG));
\end{code}
The code returns $40$ as the colength of \texttt{jgG}. There are also $5$ singular points of $C_{-10125}$ on $z = 0$ by the proof of Proposition \ref{five singular curves}. Since $\Sing(C_{-10125})$ has at most $45$ points, it is a $\mathcal{V}$-orbit by Lemma \ref{orbit size}. If the multiplicity of \texttt{jgG} was greater than $1$ at a singular point, then the number of the singular points were at most $\dfrac{45}{2}$, and this is a contradiction. Therefore, all singularities are nodes and $C_{-10125}$ has $45$ nodes.

Finally, we show that $C_{20250}$ and $C_{-10125}$ are irreducible. We assume that they are reducible and derive a contradiction. Let $C$ be $C_{20250}$ or $C_{-10125}$ and we write $C = D_{1} \sqcup \cdots \sqcup D_{n}$ where $D_{1}, \cdots, D_{n}$ are pairwise distinct integral curves. Since $C$ is reduced and nodal, by B\'{e}zout's theorem, we see
\[
\sum_{i < j} \deg D_{i} \deg D_{j} = \# \Sing C.
\]
The left hand side is $60$ or $66$ by Lemma \ref{reducible curves} (ii) and (iii). However, $\#\Sing C_{20250} = 36$ and $\# \Sing C_{-10125} = 45$. This is a contradiction.
\end{proof}

For an integral curve of degree $12$ with $n$ nodes, we calculate that its genus is $(55 - n)$ by the genus-degree formula. To summarize, we have the following.

\begin{cor}
\begin{enumerate}
\item $C_{-10125}$ is an integral, nodal and $\mathfrak{A}_{6}$-invariant projective plane curve of genus $10$.
\item $C_{20250}$ is an integral, nodal and $\mathfrak{A}_{6}$-invariant projective plane curve of genus $19$.
\end{enumerate}
\end{cor}

\section{The $\mathfrak{A}_{6}$-invariant curve of genus $10$}

Let $C$ be a curve on which the alternating group $\mathfrak{A}_{6}$ acts. The group $\mathfrak{A}_{6}$ has $2$ conjugacy classes of subgroups which are isomorphic to $\mathfrak{A}_{5}$. For $G = \mathfrak{A}_{6}$ or $\mathfrak{A}_{5}$, we consider the natural morphism $\pi_{G} : C \rightarrow C / G$. Then $\pi_{G}$ corresponds to the field extension $k(C) / k(C / G)$. We can  take the morphism $f : C / \mathfrak{A}_{5} \rightarrow C / \mathfrak{A}_{6}$ corresponding to the field extension $k(C / \mathfrak{A}_{5}) / k(C / \mathfrak{A}_{6})$.
\[
\xymatrix{
C \ar[dd]_{\pi_{\mathfrak{A}_{6}}} \ar[rd]^{\pi_{\mathfrak{A}_{5}}} & & & k(C)\ar@{-}[dd] \ar@{-}[rd] & \\
& C / \mathfrak{A}_{5} \ar[ld]^{f} & & & k(C / \mathfrak{A}_{5}) \ar@{-}[ld] \\
C / \mathfrak{A}_{6} & & & k(C / \mathfrak{A}_{6}) &
}
\]
The following lemma is an elementary consequence of Galois theory.

\begin{lemma}\label{two isom for conj icosa}
\begin{enumerate}
\item If $L$ is the Galois closure of $k(C / \mathfrak{A}_{5}) / k(C / \mathfrak{A}_{6})$, then $L = k(C)$.
\item Take an subgroup $\mathfrak{A}_{5}'$ of $\mathfrak{A}_{6}$ isomorphic to $\mathfrak{A}_{5}$. Then $\mathfrak{A}_{5}'$ is conjugate to $\mathfrak{A}_{5}$ as a subgroup of $\mathfrak{A}_{6}$ if and only if there exists an isomorphism $k(C / \mathfrak{A}_{5}') \rightarrow k(C / \mathfrak{A}_{5})$ over $k(C / \mathfrak{A}_{6})$.
\end{enumerate}
\end{lemma}

\begin{proof}
(1) Since $\mathfrak{A}_{6}$ is simple, the only normal subgroup of $\mathfrak{A}_{6}$ contained in $\mathfrak{A}_{5}$ is the trivial group $1$.

(2) If the permutation $\sigma \in \mathfrak{A}_{6}$ satisfies $\sigma^{-1} \mathfrak{A}_{5} \sigma = \mathfrak{A}_{5}'$, then it corresponds an isomorphism $k(C / \mathfrak{A}_{5}') \rightarrow k(C / \mathfrak{A}_{5})$ over $k(C / \mathfrak{A}_{6})$.

Conversely, let $\varphi$ be an isomorphism $k(C / \mathfrak{A}_{5}') \rightarrow k(C / \mathfrak{A}_{5})$ over $k(C / \mathfrak{A}_{6})$. For any isomorphism $\tau : k(C) \rightarrow k(C)$ over $k(C / \mathfrak{A}_{5})$, $\varphi^{-1} \circ \tau \circ \varphi$ is an isomorphism $k(C) \rightarrow k(C)$ over $k(C / \mathfrak{A}_{5}')$. Therefore, $\mathfrak{A}_{5}'$ is conjugate to $\mathfrak{A}_{5}$ as a subgroup of $\mathfrak{A}_{6}$.
\end{proof}

To study the curve $C$, we first describe the map $f : C / \mathfrak{A}_{5} \rightarrow C / \mathfrak{A}_{6}$. We give the ramification indices at the ramification points of $\pi_{G}$.

\begin{note}\label{remark}
For the morphism $\pi_{G} : C \rightarrow C / G$, let $P_{1}, \cdots, P_{s} \in C / G$ be the branched points and $r_{i}$ the ramification index over $P_{i}$. We may assume that $r_{1} \leq \cdots \leq r_{s}$. Then we often write $(r_{1}, \cdots, r_{s})$ as a tuple.
\end{note}

In general, the following proposition is a result of elementary calculations.

\begin{prop}\label{quotient curve condition}
Let $L := \dfrac{2g - 2}{|G|}$ for $g \geq 2$, $g'$ be the genus of $C / G$ and $s$ the number of branched points of $\pi_{G}$.
\begin{enumerate}
\item If $L < 1$, then one of the following holds:
\begin{enumerate}
\item[(i)] $g' = 1$, $s = 1$ and $r_{1} = \dfrac{1}{1 - L}$.
\item[(ii)] $g' = 0$ and $3 \leq s \leq 5$.
\end{enumerate} 
In particular, if $L < \dfrac{1}{2}$, then $g' = 0$ and $s = 3$ or $4$.

\item If $L < \dfrac{1}{6}$, then $g' = 0$, $s = 3$ and the triple of the ramification indices is one of the following:
\begin{itemize}
\item $(2, r_{2}, r_{3})$ where $3 \leq r_{2} \leq 5$ and $r_{3} = \dfrac{1}{\dfrac{1}{2} - L - \dfrac{1}{r_{2}}}$.
\item $(3, 3, 4)$. Then $L = \dfrac{1}{12}$.
\item $(3, 3, 5)$. Then $L = \dfrac{2}{15}$.
\end{itemize}
\end{enumerate}
\end{prop}

\begin{proof}
Note that $L > 0$ by the assumption $g \geq 2$. By Riemann-Hurwitz formula, we have
\begin{equation}\label{Hurwitz form}
L = 2g' - 2 + \sum_{i = 1}^{s} \left( 1 - \frac{1}{r_{i}} \right). \tag{I}
\end{equation}

(1) Assume $L < 1$. If $s = 0$, then $L = 2g' - 2$ is a positive integer, a contradiction. Thus $s > 0$. We determine $g'$ and $s$. By the equality (\ref{Hurwitz form}), we see that
\begin{equation}\label{base eq}
L - (2g' - 2) = \sum_{i = 1}^{s} \left( 1 - \frac{1}{r_{i}} \right) = s - \sum_{i = 1}^{s} \frac{1}{r_{i}}. \tag{II}
\end{equation}
Since
\begin{equation}\label{right iq}
\displaystyle\sum_{i = 1}^{s} \left(1 - \dfrac{1}{r_{i}} \right) \geq \dfrac{s}{2} \geq \dfrac{1}{2}, \tag{III}
\end{equation}
we obtain $L \geq 2g' - \dfrac{3}{2}$. If $L < 1$, then $g' = 0$ or $1$. If $L < \dfrac{1}{2}$, then $g' = 0$. 

On the other hand,
\begin{equation}\label{left iq}
s - \sum_{i = 1}^{s} \frac{1}{r_{i}} < s. \tag{IV}
\end{equation}
Hence, by the equality (\ref{base eq}), the inequality (\ref{right iq}) and (\ref{left iq}), we have
\[
L - (2g' - 2) < s \leq 2(L - (2g' - 2)).
\]
If $g' = 1$, then $L < s \leq 2L$. Since $L > 0$ and $L < 1$, we see $0 < s < 2$. Then $s = 1$, and $L = 1 - \dfrac{1}{r_{1}}$, i.e., (i) holds. If $g' = 0$, then $L + 2 < s \leq 2(L + 2)$. From $L < 1$, we see $2 < s < 6$, and (ii) holds. For $L < \dfrac{1}{2}$, we obtain $2 < s < 5$, i.e., $s = 3$ or $4$.

(2) Assume that $L < \dfrac{1}{6}$. By the statement (1), $g' = 0$ and $s = 3$ or $4$. We consider the ramification indices $(r_{1}, \cdots, r_{s})$. By the equality (\ref{Hurwitz form}), we have
\begin{equation}\label{rami form}
\sum_{i = 1}^{s} \frac{1}{r_{i}} = (s - 2) - L. \tag{V}
\end{equation}

Suppose that $s = 3$. If $r_{1} \geq 4$, then $\displaystyle\sum_{i = 1}^{3} \dfrac{1}{r_{i}} \leq \dfrac{3}{4}$ since $r_{1} \leq r_{i}$ for any $i$ by the assumption of Notation \ref{remark}. However, the right hand side of (\ref{rami form}) is greater than $\dfrac{5}{6}$. Thus $r_{1} = 2$ or $3$. Since $\dfrac{1}{r_{2}} + \dfrac{1}{r_{3}} = 1 - \dfrac{1}{r_{1}} - L$ and $r_{2} \leq r_{3}$, we see $\dfrac{2}{r_{2}} \geq 1 - \dfrac{1}{r_{1}} - L$, and 
\[
r_{1} \leq r_{2} \leq \dfrac{2}{1 - \dfrac{1}{r_{1}} - L}.
\]
Take $r_{1} = 2$. If $r_{2} = 2$, then $1 + \dfrac{1}{r_{3}} > 1 > 1 - L$, which contradicts the equality (\ref{rami form}). Then $r_{3} \geq 3$. Since $\dfrac{2}{1 - \dfrac{1}{r_{1}} - L} = \dfrac{2}{\dfrac{1}{2} - L} < 6$, we obtain $3 \leq r_{2} \leq 5$. On the other hand, take $r_{1} = 3$. Then $\dfrac{2}{1 - \dfrac{1}{r_{1}} - L} = \dfrac{2}{\dfrac{2}{3} - L} < 4$. Thus $r_{2} = 3$. Since $\dfrac{1}{r_{3}} = \dfrac{1}{3} - L$ and $0 < L < \dfrac{1}{6}$, we see $r_{3} = 4$ or $5$.

Suppose that $s = 4$. By $0 < L < \dfrac{1}{6}$ and $\displaystyle\sum_{i = 1}^{4} \dfrac{1}{r_{i}} = 2 - L$, we see $\dfrac{11}{6} < \displaystyle\sum_{i = 1}^{4} \dfrac{1}{r_{i}} < 2$. Since $(r_{1}, r_{2}, r_{3}, r_{4}) = (2, 2, 2, 2)$ does not satisfy the inequality, we have $\displaystyle\sum_{i = 1}^{4} \dfrac{1}{r_{i}} \leq \dfrac{1}{2} + \dfrac{1}{2} + \dfrac{1}{2} + \dfrac{1}{3} = \dfrac{11}{6}$. However, this is a contradiction to the former inequality. Hence, there is no $4$-tuple satisfying the equality (\ref{rami form}) for $s = 4$.
\end{proof}

This proposition implies the following lemma.

\begin{lemma}\label{genus10rami}
Let $C$ be a smooth curve of genus $10$ with a faithful $\mathfrak{A}_{6}$-action.
\begin{enumerate}
\item $C / \mathfrak{A}_{6}$ is rational, and the ramification indices at the branched points of $\pi_{\mathfrak{A}_{6}}$ are $(2, 4, 5)$.
\item $C / \mathfrak{A}_{5}$ is rational, and the ramification indices at the branched points of $\pi_{\mathfrak{A}_{5}}$ are $(4, 4, 5)$ or $(2, 2, 2, 5)$.
\end{enumerate}
\end{lemma}

\begin{proof}
(1) We see
\[
L = \frac{2 \cdot 10 - 2}{|\mathfrak{A}_{6}|} = \frac{1}{20} < \frac{1}{6}.
\]
Hence, by Proposition \ref{quotient curve condition} (2), the genus of $C / \mathfrak{A}_{6}$ is $0$, i.e., $C / \mathfrak{A}_{6}$ is rational, and the number $s$ of branched points of $\pi_{\mathfrak{A}_{6}}$ is $3$. Since $L \neq \dfrac{1}{12}$ and $L \neq \dfrac{2}{15}$, the ramification indices are $(2, r_{2}, r_{3})$ with $3 \leq r_{2} \leq 5$ and
\[
r_{3} = \frac{1}{\dfrac{9}{20} - \dfrac{1}{r_{2}}}.
\]
Since $r_{3}$ is an integer and $r_{2} \leq r_{3}$, we have $r_{2} = 4$ and $r_{3} = 5$. Thus the ramification indices at the branched points of $\pi_{\mathfrak{A}_{6}}$ are $(2, 4, 5)$.

(2) We see
\[
L = \frac{2 \cdot 10 - 2}{|\mathfrak{A}_{5}|} = \frac{3}{10} < \frac{1}{2}.
\]
By Proposition \ref{quotient curve condition} (1), the genus of $C / \mathfrak{A}_{5}$ is $0$, i.e. $C / \mathfrak{A}_{5}$ is rational, and the number $s$ of branched points of $\pi_{\mathfrak{A}_{5}}$ is $s = 3$ or $4$. On the other hand, since we have a morphism $f : C / \mathfrak{A}_{5} \rightarrow C / \mathfrak{A}_{6}$ with $\pi_{\mathfrak{A}_{6}} = f \circ \pi_{\mathfrak{A}_{5}}$, the ramification indices $r_{i}$ are divisors of $2$, $4$ or $5$, i.e., $r_{i} = 2$, $4$ or $5$. We recall the equation (\ref{rami form}) in the proof of Proposition \ref{quotient curve condition}:
\[
\sum_{i = 1}^{s} \frac{1}{r_{i}} = (s - 2) - \frac{3}{10}.
\]

Suppose that $s = 3$. Then
\[
\sum_{i = 1}^{3} \frac{1}{r_{i}} = \frac{7}{10}.
\]
If $r_{1} = 2$, then $r_{2} = 2$, $4$ or $5$ and we have $\dfrac{1}{r_{1}} + \dfrac{1}{r_{2}} + \dfrac{1}{r_{3}} > \dfrac{1}{2} + \dfrac{1}{5} = \dfrac{7}{10}$. This is a contradiction. Assume that $r_{1} = 4$. Then $\dfrac{1}{r_{2}} + \dfrac{1}{r_{3}} = \dfrac{9}{20}$, and $\dfrac{20}{9} < r_{2} \leq \dfrac{40}{9}$. Thus $(r_{1}, r_{2}, r_{3}) = (4, 4, 5)$. Assume that $r_{1} = 5$. Then $r_{1} = r_{2} = r_{3} = 5$, but then $\displaystyle\sum_{i = 1}^{3} \dfrac{1}{r_{i}} = \dfrac{3}{5} \neq \dfrac{7}{10}$.

Suppose that $s = 4$. Then
\[
\sum_{i = 1}^{4} \frac{1}{r_{i}} = \frac{17}{10}.
\]
If $r_{2} > 2$, then $r_{2} \geq 4$ and $\displaystyle\sum_{i = 1}^{4} \dfrac{1}{r_{i}} \leq \dfrac{1}{2} + \dfrac{1}{4} + \dfrac{1}{4} + \dfrac{1}{4} < \dfrac{17}{10}$. Thus $r_{1} = r_{2} = 2$. We obtain $\dfrac{1}{r_{3}} + \dfrac{1}{r_{4}} = \dfrac{7}{10}$, and $\dfrac{10}{7} < r_{3} \leq \dfrac{20}{7}$. Hence, $r_{3} = 2$, and we see $r_{4} = 5$. Therefore, the tuple of ramification indices is $(2, 2, 2, 5)$.
\end{proof}

By Lemma \ref{genus10rami}, $f$ is a rational surjection $\mathbb{P}^{1} \rightarrow \mathbb{P}^{1}$, and we can write it as a rational function.

\begin{lemma}\label{genus10f}
Let $C$ be a smooth curve of genus $10$ with a faithful $\mathfrak{A}_{6}$-action. Take the coordinate on $C / \mathfrak{A}_{6} \cong \mathbb{P}^{1}$ for which $1 \in \mathbb{P}^{1}$ is the branched point of index $2$, $0 \in \mathbb{P}^{1}$ is the branched point of index $4$ and $\infty \in \mathbb{P}^{1}$ is the branched point of index $5$. Then the following holds.
\begin{enumerate}
\item If the ramification indices at the branched points of $\pi_{\mathfrak{A}_{5}}$ are $(4, 4, 5)$, then we can write
\[
f(w) = \frac{w^{4}(w^{2} + 4w + 20)}{256(w - 1)}
\]
in a suitable coordinate on $C / \mathfrak{A}_{5}$.
\item If the ramification indices at the branched points of $\pi_{\mathfrak{A}_{5}}$ are $(2, 2, 2, 5)$, then we can write
\[
f(w) = \frac{\lambda_{0} w^{4}((256\lambda_{0} - 25)w - 50)^{2}}{2500(w - 1)}
\]
in a suitable coordinate on $C / \mathfrak{A}_{5}$ where $\lambda_{0}$ is a solution of the quadratic equation
\[
442368 \Lambda^{2} - 72000 \Lambda + 3125 = 0.
\]
\end{enumerate}
\end{lemma}

\begin{remark}
The possibility of the form (1) is later excluded.
\end{remark}

\begin{proof}
Let $Q_{1}^{(r)}, Q_{2}^{(r)}, \cdots$ denote the branch points of $\pi_{\mathfrak{A}_{5}}$ of index $r$. For each $r$ and $i$,  the ramification index of $\pi_{\mathfrak{A}_{6}}$ at $f(Q_{i}^{(r)})$ is divisible by $r$.

(1) The branch points of $\pi_{\mathfrak{A}_{5}}$ are $Q_{1}^{(4)}$, $Q_{2}^{(4)}$ and $Q_{1}^{(5)}$. Since the only ramification index of $\pi_{\mathfrak{A}_{6}}$ divisible by $4$ (resp. $5$) is $4$ (resp. $5$), the map $f$ sends $Q_{1}^{(4)}$ and $Q_{2}^{(4)}$ to $0$ and $Q_{1}^{(5)}$ to $\infty$. Since $\deg f = [k(C / \mathfrak{A}_{5}) : k(C / \mathfrak{A}_{6})] = [\mathfrak{A}_{6} : \mathfrak{A}_{5}] = 6$, there are $1$ ramified point of $f$ of index $4$ over $0 \in \mathbb{P}^1$, $1$ ramified point of index $5$ over $\infty \in \mathbb{P}^1$ and $3$ ramified points of index $2$ over $1 \in \mathbb{P}^1$. By taking a suitable projective transformation of  $C / \mathfrak{A}_{5} \cong \mathbb{P}^{1}$, we may assume that $f^{-1}(0) = \left\{ 0, Q_{1}^{(4)}, Q_{2}^{(4)} \right\}$ and $f^{-1}(\infty) = \left\{ \infty, Q_{1}^{(5)} = 1 \right\}$. Then we can write
\[
f(w) = \frac{w^{4}(\kappa w^{2} + \lambda w + \mu)}{w - 1}
\]
for some $\kappa$, $\lambda$, $\mu \in \mathbb{C}$. On the other hand, the numerator of
\[
f(w) - 1 = \frac{\kappa w^6 + \lambda w^5 + \mu w^4 - w + 1}{w - 1}
\]
decomposes as $(\alpha w^{3} + \beta w^{2} + \gamma w + \delta)^{2}$ for $\alpha$, $\beta$, $\gamma$, $\delta \in \mathbb{C}$ since $f^{-1}(1)$ consists of $3$ ramified points of index $2$. By equating the corresponding coefficients of powers of $w$,
\[
\left\{
\begin{array}{rcl}
1 & = & \delta^{2}, \\
-1 & = & 2\gamma\delta, \\
0 & = & 2\beta\delta + \gamma^{2}, \\
0 & = & 2\alpha\delta + 2\beta\gamma, \\
\mu & = & 2\alpha\gamma + \beta^{2}, \\
\lambda & = & 2\alpha\beta, \\
\kappa & = & \alpha^{2}.
\end{array}
\right.
\]
We see $\delta = \pm 1$ by the first equality, $\gamma = \mp\dfrac{1}{2}$ by the second equality, $\beta = \mp\dfrac{1}{8}$ by the third equality and $\alpha = \mp\dfrac{1}{16}$ by the fourth equality. Therefore, by the $5$-th to $7$-th equalities,
\[
(\kappa, \lambda, \mu) = \left( \dfrac{1}{256}, \dfrac{1}{64}, \dfrac{5}{64} \right),
\]
and $f(w)$ is given.

(2) The branch points of $\pi_{\mathfrak{A}_{5}}$ are $Q_{1}^{(2)}$, $Q_{2}^{(2)}$, $Q_{3}^{(2)}$ and $Q_{1}^{(5)}$. The map $f$ sends $Q_{1}^{(5)}$ to $\infty$ and $Q_{i}^{(2)}$ to $0$ or $1$ for each $i$. Thus $f$ has $1$ ramified point of index $5$ over $\infty$ and $Q_{1}^{(2)}, Q_{2}^{(2)}$ and $Q_{3}^{(2)}$ satisfy one of the following:
\begin{enumerate}
\item[(i)] All points are sent to $0$.
\item[(ii)] $2$ points are sent to $0$ and $1$ point is sent to $1$.
\item[(iii)] $1$ point is sent to $0$ and $2$ points are sent to $1$.
\item[(iv)] All points are sent to $1$.
\end{enumerate}
If $Q_{i}^{(2)}$ is sent to $1$, then it is an unramified point for $f$. Since any point which is sent to $1$ and different from $Q_{i}^{(2)}$ is a ramified point of index $2$ for $f$ and $\deg f = 6$ is even, (ii) and (iv) are impossible.

Assume the case (i). Then there are $3$ ramified points $Q_{1}^{(2)}, Q_{2}^{(2)}, Q_{3}^{(2)}$ for $f$ of index $2$ over $0$ and $3$ ramified points of index $2$ over $1$. By taking a suitable projective transformation on $C / \mathfrak{A}_{5} \cong \mathbb{P}^{1}$, we may assume that $f^{-1}(0) = \left\{ Q_{1}^{(2)} = 0, Q_{2}^{(2)}, Q_{3}^{(2)} \right\}$ and $f^{-1}(\infty) = \left\{ \infty, Q_{1}^{(5)} = 1 \right\}$. Then we can write
\[
f(w) = \frac{w^2 (\kappa w^2 + \lambda w + \mu)^2}{w - 1}
\]
for some $\kappa$, $\lambda$, $\mu \in \mathbb{C}$. On the other hand, since $f^{-1}(1)$ consists $3$ ramified point $Q_{1}^{(2)}$, $Q_{2}^{(2)}$, $Q_{3}^{(2)}$ of index $2$, the numerator of
\[
f(w) - 1 = \frac{\kappa^2 w^6 + 2 \kappa\lambda w^5 + (2\kappa\mu + \lambda^2) w^4 + 2 \lambda\mu w^3 + \mu^2 w^2 - w + 1}{w - 1}
\]
decomposes as $(\alpha w^3 + \beta w^2 + \gamma w + \delta)^2$ for some $\alpha \neq 0$, $\beta$, $\gamma$, $\delta \in \mathbb{C}$. Then we may assume that $\alpha = \kappa$. We obtain
\[
\left\{
\begin{array}{rcl}
1 & = & \delta^{2}, \\
-1 & = & 2\gamma\delta, \\
\mu^2 & = & 2\beta\delta + \gamma^{2}, \\
2 \lambda \mu & = & 2\alpha\delta + 2\beta\gamma, \\
2 \alpha \mu + \lambda^2 & = & 2\alpha\gamma + \beta^{2}, \\
2 \alpha \lambda & = & 2\alpha\beta,
\end{array}
\right.
\]
by equating the corresponding coefficients. Since we see $\delta = \pm 1$, $\gamma = \mp \dfrac{1}{2}$ and $\beta = \lambda$ by the first, second and sixth equations,
\[
\left\{
\begin{array}{rcl}
\mu^2 & = & \pm 2\beta + \dfrac{1}{4}, \\
2 \beta \mu & = & \pm 2\alpha \mp \beta, \\
2 \alpha \mu + \beta^2 & = & \mp \alpha + \beta^{2}.
\end{array}
\right.
\]
By the last equation, we see $\alpha(2\mu \pm 1) = 0$. Since $\alpha \neq 0$, we have $\mu = \mp \dfrac{1}{2}$. Thus $\beta = 0$ by the first equation, and $\alpha = 0$ by the second equation. This is a contradiction.

Therefore, we have (iii). Apart from $Q_{i}^{(2)}$, there is $1$ ramified point of index $4$ over $0$ and there are $2$ ramified points of index $2$ over $1$. By taking a suitable projective transformation on $C / \mathfrak{A}_{5} \cong \mathbb{P}^{1}$, we may assume that $f^{-1}(0) = \left\{ 0, Q_{1}^{(2)} \right\}$ and $f^{-1}(\infty) = \left\{ \infty, Q_{1}^{(5)} = 1 \right\}$. Then we can write
\[
f(w) = \frac{w^{4}(\kappa w + \lambda)^{2}}{w - 1}
\]
for $\kappa$, $\lambda \in \mathbb{C}$. On the other hand, since $f^{-1}(1)$ consists of $Q_{2}^{(2)}, Q_{3}^{(2)}$ and $2$ ramified points of index $2$, the numerator of
\[
f(w) - 1 = \frac{\kappa^2 w^6 + 2\kappa\lambda w^5 + \lambda^2 w^4 - w + 1}{w - 1}
\]
decomposes as $(w^{2} + \alpha w + \beta)^{2}(\gamma w^{2} + \delta w + \varepsilon)$ for some $\alpha \neq 0$, $\beta$, $\gamma$, $\delta$, $\varepsilon \in \mathbb{C}$. By equating the corresponding coefficients for powers of $w$,
\begin{equation}\label{genus10func}
\left\{
\begin{array}{rcl}
1 & = & \beta^{2} \varepsilon, \\
-1 & = & 2\alpha\beta\varepsilon + \beta^{2}\delta, \\
0 & = & (\alpha^{2} + 2\beta)\varepsilon + 2\alpha\beta\delta + \beta^{2}\gamma, \\
0 & = & 2\alpha\varepsilon + (\alpha^{2} + 2\beta)\delta + 2\alpha\beta\gamma, \\
\lambda^{2} & = & \varepsilon + 2\alpha\delta + (\alpha^{2} + 2\beta)\gamma, \\
2\kappa\lambda & = & \delta + 2\alpha\gamma, \\
\kappa^{2} & = & \gamma.
\end{array} \tag{$\ast$}
\right.
\end{equation}
We solve this system of equations with SINGULAR. We put \texttt{a} $= \alpha$, \texttt{b} $= \beta$, \texttt{c} $= \gamma$, \texttt{d} $= \delta$, \texttt{e} $= \varepsilon$, \texttt{k} $= \kappa$ and \texttt{l} $= \lambda$. Let \texttt{I} be the ideal corresponding to (\ref{genus10func}). We check the elimination ideal \texttt{EI} of \texttt{I}.
\begin{code}{}{}
ring R = 0,(a,b,c,d,e,k,l),dp;
poly p0 = b2e - 1;
poly p1 = 2abe + b2d + 1;
poly p2 = a2e + 2be + 2abd + b2c;
poly p3 = 2ae + a2d + 2bd + 2abc;
poly p4 = e + 2ad + a2c + 2bc - l2;
poly p5 = d + 2ac - 2kl;
poly p6 = c - k2;
ideal I = p0,p1,p2,p3,p4,p5,p6;

//Eliminate a,b,c,d and e from I.
//Take EI to be this elimination ideal.
ideal EI = eliminate(I,abcde);
EI;
\end{code}
The code returns generators of \texttt{EI}.
\begin{result}{}{}
EI[1]=3456kl+1152l2-125
EI[2]=54k2-9kl+l2
EI[3]=256l3+50k-25l
\end{result}
By this third line, the pair $(\kappa, \lambda)$ satisfies $50\kappa = -\lambda(256\lambda^{2} - 25)$. Thus we can write
\[
f(w) = \frac{\lambda^{2}w^{4}((256\lambda^{2} - 25)w - 50)^{2}}{2500(w - 1)}.
\]
Finally, we eliminate \texttt{k} from \texttt{EI} to give the value of $\lambda^{2}$.
\begin{code}{}{}
eliminate(EI, k);
\end{code}
The line returns 
\[
442368 \lambda^{4} - 72000 \lambda^{2} + 3125.
\]
This means that $\lambda^{2}$ is a root of $442368 \Lambda^{2} - 72000 \Lambda + 3125$. Therefore, we have $f(w)$ as in the statement (2).
\end{proof}

Now we exclude the possibility of Lemma \ref{genus10f} (1) by considering the Galois closure of the corresponding extension $M / K$.

\begin{lemma}\label{genus10not(1)}
Let $\varphi$ be a surjective morphism $\mathbb{P}^{1} \rightarrow \mathbb{P}^{1}$ given by
\[
\varphi(w) = \dfrac{w^{4}(w^{2} + 4w + 20)}{256(w - 1)},
\]
$M / K$ the corresponding field extension and $L$ its Galois closure. Then $\Gal(L / K) \not\cong \mathfrak{A}_{6}$.
\end{lemma}

\begin{proof}
Since the fields $M$ and $K$ correspond to $\mathbb{P}^{1}$, we may assume that $M = \mathbb{C}(w)$ and $K = \mathbb{C}(t)$ where $w$ and $t$ are transcendental over $\mathbb{C}$. Then the pull back $\varphi^{*}$ is an injective homomorphism $\mathbb{C}(t) \rightarrow \mathbb{C}(w)$ given by $\varphi(w) = t$. Thus we obtain
\[
p(w) := w^{6} + 4 w^{5} + 20 w^{4} + (-256 t) w + (256 t) = 0.
\]
Suppose that the factorization of $p(w)$ over $L$ is $(w - \phi_{1}) \cdots (w - \phi_{6})$. The Galois group $\Gal(L / K)$ is a subgroup of the permutation group of $\phi_{1}, \cdots, \phi_{6}$. We will give an intermediate field $F$ of $L / K$ which has degree $2$ over $K$. We can check that the discriminant $\Delta$ of $p(w)$ is
\[
\begin{array}{rcl}
\Delta & = & R\left(p, \dfrac{\partial p}{\partial w}\right) \\
& & \\
& = & \left| \begin{array}{ccccccccccc}
1 & 4 & 20 & 0 & 0 & -256t & 256t & 0 & 0 & 0 & 0 \\
0 & 1 & 4 & 20 & 0 & 0 & -256t & 256t & 0 & 0 & 0 \\
&&&&& \vdots &&&&&\\
0 & 0 & 0 & 0 & 1 & 4 & 20 & 0 & 0 & -256t & 256t\\
6 & 20 & 80 & 0 & 0 & -256t & 0 & 0 & 0 & 0 & 0 \\
&&&&& \vdots &&&&&\\
0 & 0 & 0 & 0 & 0 & 6 & 20 & 80 & 0 & 0 & -256t \\
\end{array} \right| \\
& & \\
& = & -879609302220800000 t^3 + 2638827906662400000 t^4 \\
& & -2638827906662400000 t^5 + 879609302220800000 t^6.
\end{array}
\]
Let $\psi := \displaystyle\prod_{i < j} (\phi_{j} - \phi_{i}) \in L$, and then $\psi^{2} = \Delta$. Since the order of $\Delta$ in $t$ is $3$, which is odd, we have $\psi \not\in K$, and $F := K(\psi)$ is an extension field of degree $2$ over $K$. Therefore, the Galois group $\Gal(L / F)$ is a subgroup of $\Gal(L / K)$ of index $2$. Since $\mathfrak{A}_{6}$ has no subgroup of index $2$, $\Gal(L / K)$ is not isomorphic to $\mathfrak{A}_{6}$.
\end{proof}

\begin{theorem}
Let $C$ be a smooth projective curve of genus $10$ with a faithful $\mathfrak{A}_{6}$-action. Then the morphism $f : C / \mathfrak{A}_{5} \rightarrow C / \mathfrak{A}_{6}$ can be written as
\[
f(w) = \frac{\lambda_{0} w^{4}((256\lambda_{0} - 25)w - 50)^{2}}{2500(w - 1)}
\]
in a suitable coordinate on $C / \mathfrak{A}_{5}$ and $C / \mathfrak{A}_{6}$ where $\lambda_{0}$ is a solution of the quadratic equation $442368 \Lambda^{2} - 72000 \Lambda + 3125 = 0$.

In particular, an $\mathfrak{A}_{6}$-invariant curve of genus $10$ is unique up to isomorphism.
\end{theorem}

\begin{proof}
By Lemma \ref{genus10f} and Lemma \ref{genus10not(1)}, for non-conjugate icosahedral subgroups $\mathfrak{A}_{5}$ and $\mathfrak{A}_{5}'$, the corresponding morphism $f : C / \mathfrak{A}_{5} \rightarrow \mathbb{P}^{1}$ and $f' : C / \mathfrak{A}_{5}' \rightarrow \mathbb{P}^{1}$ are written in the above form. (Note that the base $\mathbb{P}^{1}$'s can be identified since the branched points are same.) Since $f$ and $f'$ are not conjugate by Lemma \ref{two isom for conj icosa}(2), they have to correspond to the $2$ choice for $\lambda_{0}$. Thus the Galois closure gives $C$ for any of $\lambda_{0}$.
\end{proof}

\section{The $\mathfrak{A}_{6}$-invariant curve of genus $19$}

In this section, let $C$ be an $\mathfrak{A}_{6}$-invariant curve of genus $19$. (Note that Lemma \ref{two isom for conj icosa} and Proposition \ref{quotient curve condition} also hold in the case of genus $19$.) Fix an icosahedral group $\mathfrak{A}_{5} < \mathfrak{A}_{6}$. Recall that $\pi_{G}$ is the natural morphism $C \rightarrow C / G$ for $G = \mathfrak{A}_{6}$ or $\mathfrak{A}_{5}$.

\begin{lemma}\label{genus19rami}
Let $C$ be a smooth curve of genus $19$ with a faithful $\mathfrak{A}_{6}$-action.
\begin{enumerate}
\item $C / \mathfrak{A}_{6}$ is rational, and the ramification indices at the branched points of $\pi_{\mathfrak{A}_{6}}$ are $(2, 5, 5)$.
\item $C / \mathfrak{A}_{5}$ is rational, and the ramification indices at the branched points of $\pi_{\mathfrak{A}_{5}}$ are $(2, 2, 5, 5)$.
\end{enumerate}
\end{lemma}

\begin{proof}
(1) We see
\[
L = \frac{2 \cdot 19 - 2}{|\mathfrak{A}_{6}|} = \frac{1}{10} < \frac{1}{6},
\]
and $L \neq \dfrac{1}{12}$ and $L \neq \dfrac{2}{15}$. By the Proposition \ref{quotient curve condition} (2), $C / \mathfrak{A}_{6}$ is rational and the ramification indices is $(2, r_{2}, r_{3})$ with $3 \leq r_{2} \leq 5$ and
\[
r_{3} = \frac{1}{\dfrac{2}{5} - \dfrac{1}{r_{2}}}.
\]
Since $r_{3}$ is an integer and $r_{2} \leq r_{3}$, we have $(r_{2}, r_{3}) = (3, 15)$ or $(5, 5)$. However, $\mathfrak{A}_{6}$ has no subgroup of order $15$. Thus the ramification indices of $\pi_{\mathfrak{A}_{6}}$ are $(2, 5, 5)$.

(2) We see
\[
L = \frac{2 \cdot 19 - 2}{|\mathfrak{A}_{5}|} = \frac{3}{5} < 1.
\]
By the Proposition \ref{quotient curve condition} (1), the genus $g$ of $C / \mathfrak{A}_{5}$ and the number $s$ of the branched points of $\mathfrak{A}_{5}$ satisfy one of the following:
\begin{enumerate}
\item[(i)] $g = 1$ and $s = 1$,
\item[(ii)] $g = 0$ and $3 \leq s \leq 5$.
\end{enumerate}
If the condition (i) holds, then the ramification index is
\[
r_{1} = \frac{1}{1-\dfrac{3}{5}} = \frac{5}{2}.
\]
However, this is not an integer. Hence, we have (ii) and $C / \mathfrak{A}_{5}$ is rational. We recall the equation (\ref{rami form}) in the proof of the Proposition \ref{quotient curve condition}:
\[
\sum_{i = 1}^{s} \frac{1}{r_{i}} = (s - 2) - \frac{3}{5}.
\]
On the other hand, since we have a morphism $f : C / \mathfrak{A}_{5} \rightarrow C / \mathfrak{A}_{6}$ with $\pi_{\mathfrak{A}_{6}} = f \circ \pi_{\mathfrak{A}_{5}}$, the ramification indices $r_{i}$ are divisors of $2$ or $5$, i.e., $r_{i} = 2$ or $5$. Thus there is a positive integer $k$ such that $\displaystyle\sum_{i = 1}^{s} \dfrac{1}{r_{i}} = \dfrac{1}{2} k + \dfrac{1}{5} (s - k)$. Since this is equal to $(s - 2) - \dfrac{3}{5}$, we see $k = \dfrac{1}{3}(8s - 26)$. If $s = 3$ or $5$, then $k$ is not an integer. Therefore, $s = 4$ and $(r_{1}, r_{2}, r_{3}, r_{4}) = (2, 2, 5, 5)$.
\end{proof}

By Lemma \ref{genus19rami} (1) and (2), $C / \mathfrak{A}_{6}$ and $C / \mathfrak{A}_{5}$ are rational. Thus the corresponding morphism $f : C / \mathfrak{A}_{5} \rightarrow C / \mathfrak{A}_{6}$ can be identified with a surjective rational map $\mathbb{P}^{1} \rightarrow \mathbb{P}^{1}$.

\begin{lemma}\label{genus19f}
Let $C$ be a smooth curve of genus $19$ with a faithful $\mathfrak{A}_{6}$-action. Take the coordinate on $C / \mathfrak{A}_{6}$ for which $1$ is the branched point of index $2$ and $0$ and $\infty$ are two branched points of index $5$. In a suitable coordinate, the following holds:
\begin{enumerate}
\item $f(w) = \dfrac{w^{5}(w + 4)}{64(w - 1)}$,
\item $f(w) = \dfrac{w^{5}((-783\lambda_{0} + 64)w + 972\lambda_{0})}{972(w - 1)}$, where $\lambda_{0}$ is a solution of the quadratic equation $11664\Lambda^{2} -1647\Lambda + 64 = 0$.
\end{enumerate}
\end{lemma}

\begin{remark}
Similar to Section 3, the possibility of the form (1) is later excluded.
\end{remark}

\begin{proof}
Let $Q_{1}^{(r)}, Q_{2}^{(r)}, \cdots$ denote the branch points of $\pi_{\mathfrak{A}_{5}}$ of index $r$. We note that the ramification index of $\pi_{\mathfrak{A}_{6}}$ at $f(Q_{i}^{(r)})$ is divisible by $r$ for each $r$ and $i$. The branch points of $\pi_{\mathfrak{A}_{5}}$ are $Q_{1}^{(2)}, Q_{2}^{(2)}, Q_{1}^{(5)}$ and $Q_{2}^{(5)}$ by Lemma \ref{genus19rami}, and the map $f$ sends $Q_{i}^{(2)}$ to $1$ for any $i$ and $Q_{i}^{(5)}$ to $0$ or $\infty$ for each $i$.

Assume that both $Q_{1}^{(5)}$ and $Q_{2}^{(5)}$ are sent to $0$ (resp. $\infty$). They are unramified points for $f$. Thus $f^{-1}(0)$ (resp. $f^{-1}(\infty)$) has another point $Q$. Since the ramification index at $Q$ is equal to the ramification index of $0$ (resp. $\infty$), i.e., $5$, we see $\deg f^{*}(0) \geq 7$ (resp. $\deg f^{*}(\infty) \geq 7$). However, $\deg f = [\mathfrak{A}_{6} : \mathfrak{A}_{5}] = 6$. This is a contradiction.

By taking a suitable projective transformation of $\mathbb{P}^{1}$, we may assume that $f$ sends $Q_{1}^{(5)} = 0$ to $0$, $Q_{2}^{(5)} = \infty$ to $\infty$ and $1$ to $\infty$. Then we can write
\[
f(w) = \frac{w^{5}(\kappa w + \lambda)}{w - 1}
\]
where $\kappa, \lambda \in \mathbb{C}$ and $\kappa \neq 0$. On the other hand, since there are $2$ ramified points of $f$ of index $2$ at $1$, the numerator of 
\[
f(w) - 1 = \frac{\kappa w^6 + \lambda w^5 - w + 1}{w - 1}
\]
decomposes as $(w^{2} + \alpha w + \beta)^{2}(\kappa w^{2} + \gamma w + \delta)$ for $\alpha \neq 0, \beta, \gamma, \delta \in \mathbb{C}$. By equating the corresponding coefficients for $w$,
\begin{equation}\label{genus19func}
\left\{
\begin{array}{rcl}
1 & = & \beta^{2} \delta, \\
-1 & = & 2\alpha\beta\delta + \beta^{2}\gamma, \\
0 & = & (\alpha^{2} + 2\beta)\delta + 2\alpha\beta\gamma + \beta^{2}\kappa, \\
0 & = & 2\alpha\delta + (\alpha^{2} + 2\beta)\gamma + 2\alpha\beta\kappa, \\
0 & = & \delta + 2\alpha\gamma + (\alpha^{2} + 2\beta)\kappa, \\
\lambda & = & \gamma + 2\alpha\kappa. \\
\end{array} \tag{$\ast$}
\right.
\end{equation}
We use SINGULAR to solve this system of equations. We put \texttt{a} $= \alpha$, \texttt{b} $= \beta$, \texttt{c} $= \gamma$, \texttt{d} $= \delta$, \texttt{k} $= \kappa$ and \texttt{l} $= \lambda$. Let \texttt{I} be the ideal corresponding to (\ref{genus19func}). We check the elimination ideal \texttt{EI} of \texttt{I}.
\begin{code}{}{}
ring R = 0,(a,b,c,d,k,l),dp;

poly p0 = b2d - 1;
poly p1 = 2abd + b2c + 1;
poly p2 = a2d + 2bd + 2abc + b2k;
poly p3 = 2ad + a2c + 2bc + 2abk;
poly p4 = d + 2ac + a2k + 2bk;
poly p5 = c + 2ak - l;
ideal I = p0,p1,p2,p3,p4,p5;

//Eliminate a,b,c and d from I.
//Take EI to be this elimination ideal.
ideal EI = eliminate(I,abcd);
EI;
\end{code}
This code returns generators of \texttt{EI}.
\begin{result}{}{}
EI[1]=216l2-954k-799l+64
EI[2]=46656kl-864l2+166896k+136801l-11200
EI[3]=3024k2+2448kl-l2-200k
\end{result}
Its first generator gives the equation
\begin{equation}\label{genus19kappa}
216\lambda^{2} - 799 \lambda + 64 = 954\kappa. \tag{$\ast\ast$}
\end{equation}
Furthermore, eliminate \texttt{k} from \texttt{EI}.
\begin{code}{}{}
ideal J = eliminate(EI,k);
J;
factorize(J[1]);
\end{code}
The code returns the polynomial of $\lambda$
\[
186624\lambda^{3} - 38016\lambda^{2} + 2671\lambda - 64,
\]
and it is factored into $(16\lambda - 1)(11664\lambda^{2} -1647\lambda + 64)$. Therefore, $\lambda = \dfrac{1}{16}$ or a solution of $11664\lambda^{2} -1647\lambda + 64 = 0$. By the equality (\ref{genus19kappa}), if $\lambda = \dfrac{1}{16}$, then $\kappa = \dfrac{1}{64}$. Otherwise, we see $\kappa = \dfrac{1}{954} \left( - \dfrac{1537}{2}\lambda + \dfrac{1696}{27} \right) = \dfrac{-783\lambda + 64}{972}$.
\end{proof}

\begin{remark}\label{transform f}
For the rational function $f(w) = \dfrac{w^5(\kappa w + \lambda)}{w - 1}$, we calculate
\[
f\left(-\frac{\lambda}{\kappa w}\right) = \frac{\left(\frac{\lambda}{\kappa w}\right)^5 \left(\kappa\frac{\lambda}{\kappa w} + \lambda \right)}{\frac{\lambda}{\kappa w} - 1} = \frac{\lambda^6 (w - 1)}{\kappa^4 w^5 (\kappa w + \lambda)} = \frac{\lambda^6}{\kappa^4} \dfrac{1}{f(w)}.
\]
Since $\dfrac{\lambda^3}{\kappa^2} = -1$ for $(\kappa, \lambda) = \left( \dfrac{-783\lambda_{0} + 64}{972}, \lambda_{0} \right)$, we have $f\left(-\dfrac{\lambda}{\kappa w}\right) = \dfrac{1}{f(w)}$.
\end{remark}

Now we exclude the possibility of Lemma \ref{genus19f} (1) by considering the Galois closure of the corresponding extension $M / K$.

\begin{lemma}\label{genus19not(1)}
Let $\varphi$ be a surjective morphism $\mathbb{P}^{1} \rightarrow \mathbb{P}^{1}$ given by
\[
\varphi(w) = \dfrac{w^{5}(w + 4)}{64(w - 1)},
\]
$M / K$ the corresponding field extension and $L$ its Galois closure. Then $\Gal(L / K) \not\cong \mathfrak{A}_{6}$.
\end{lemma}

\begin{remark}
Since the discriminant of $p(w)$ is square, we can not prove this in the same way as in Lemma \ref{genus10not(1)}. Here we use the \emph{resolvent} $f_{15}(x)$ used in \cite[Section 3]{Hagedorn}.

Take a polynomial $\omega_{0} \in K[\rho_{1}, \cdots, \rho_{6}]$. Under the action of $\mathfrak{S}_{6}$ on $\rho_{1}, \cdots, \rho_{6}$, let $G$ be the stabilizer subgroup $\Stab(\omega_{0})$. Let $\Omega$ be the orbit of $\omega_{0}$ under the action $\mathfrak{S}_{6}$. Then the polynomial $f(x) = \displaystyle\prod_{\omega \in \Omega} (x - \omega)$ is called a resolvent for $G$. For example, $f_{2}(x) = x^2 - \Delta$ where $\Delta$ is the discriminant is a resolvent for $\mathfrak{A}_{6}$.
\end{remark}

\begin{proof}
Since the fields $M$ and $K$ correspond to $\mathbb{P}^{1}$, we can write $M = \mathbb{C}(w)$ and $K = \mathbb{C}(t)$ where $w$ and $t$ are transcendental over $\mathbb{C}$. The pull back $\varphi^{*}$ is an injective homomorphism $\mathbb{C}(t) \rightarrow \mathbb{C}(w)$ given by $t = \varphi(w)$. Thus we see
\[
p(w) := w^{6} + 4w^{5} - 64 t w + 64 t = 0.
\]
Assume that the factorization of $p(w)$ over $L$ is $(w - \rho_{1}) \cdots (w - \rho_{6})$. Let $S_{i}$ be the elementary symmetric polynomial of degree $i$ for $\rho_{1}, \cdots, \rho_{6}$. Then their values are
\begin{equation}\label{list of coefficient}
\begin{array}{rcl}
S_{1} &=& -4,\\
S_{2} &=& 0,\\
S_{3} &=& 0,\\
S_{4} &=& 0,\\
S_{5} &=& 64t,\\
S_{6} &=& 64t.
\end{array}\tag{$\ast$}
\end{equation}
The Galois group $\Gal(L / K)$ is a subgroup of the permutation group of $\rho_{1}, \cdots, \rho_{6}$, which we identify with $\mathfrak{S}_{6}$.

Let $\omega_{0} = \rho_{1}\rho_{2} + \rho_{3}\rho_{4} + \rho_{5}\rho_{6}$. Then the orbit $\Omega = \mathfrak{S}_{6} \cdot \omega_{0}$ is of order $15$, and $\mathfrak{A}_{6}$ acts transitively on $\Omega$. To see this, note that there is a permutation $\sigma \in \mathfrak{S}_{6}$ with $\sigma \cdot \omega_{0} = \omega$ for any $\omega \in \Omega$. If $\sigma$ is even, then $\sigma \in \mathfrak{A}_{6}$. Suppose that $\sigma$ is odd. Since $(1 2) \cdot \omega_{0} = \omega_{0}$, we obtain $(\sigma (1 2)) \cdot \omega_{0} = \sigma \cdot ((1 2) \cdot \omega_{0}) = \sigma \cdot \omega_{0} = \omega$. Thus $\sigma (1 2)$ is an even permutation which sends $\omega_{0}$ to $\omega$.

Assume that $\Gal(L / K) = \mathfrak{A}_{6}$. We consider the polynomial of degree $15$
\[
f_{15}(x) = \prod_{\omega \in \Omega} (x - \omega(\rho_{1}, \cdots, \rho_{6})) = \sum_{i = 0}^{15} a_{i} x^{15 - i} \in (\mathbb{C}(t))[x]
\]
which decomposes over an intermediate field $N := K( \left\{ \omega(\rho_{1}, \cdots, \rho_{6}) \mid \omega \in \Omega \right\} )$ of $L / K$. Then, since $\mathfrak{A}_{6}$ acts transitively on $\Omega$, it also acts transitively on the roots of $f_{15}(x)$. The polynomial $f_{15}(x)$ must be a power of an irreducible polynomial over $K$. 

We calculate $f_{15}(x)$ from the coefficients of $p(w)$ with SINGULAR. We put \texttt{r(1)}, $\cdots$, \texttt{r(6)} to be variables corresponding to $\rho_{1}, \cdots, \rho_{6}$. First, we give the set $\Omega$ as \texttt{Omg} with the following code.
\begin{code}{}{}
LIB "sets.lib";
LIB "ellipticcovers.lib";

ring R1 = 0,r(1..6),dp;

list S6 = permute(list(1,2,3,4,5,6));
//S6 is the list of permutations of 1, 2, 3, 4, 5 and 6.

poly omg0 = r(1)*r(2)+r(3)*r(4)+r(5)*r(6);
Set Omg = list();
for(int i = 1;i<=720;i=i+1)
{
map sigma = R1,r(S6[i][1]),r(S6[i][2]),r(S6[i][3]),r(S6[i][4]),
r(S6[i][5]),r(S6[i][6]);
Omg=addElement(Omg, sigma(omg0));
}

// Put Omglist be a list of elements in Omg.
list Omglist = Omg.elements;
\end{code}
Next, we calculate the polynomial $f_{15}(x)$. We take \texttt{S[i]} to be $S_{i}$ and \texttt{I} be the ideal generated by the polynomials corresponding to (\ref{list of coefficient}) and $f_{15}(x)$. Then we take \texttt{EI} to be the ideal obtained by eliminating \texttt{r(1)}, $\cdots$, \texttt{r(6)} from \texttt{I}.
\begin{code}{}{}
LIB "poly.lib";
ring R2 =  0,(r(1..6),t,x),dp;
map phi = R1,r(1),r(2),r(3),r(4),r(5),r(6);

poly f15 = 1;
for(int j=1;j<=15;j=j+1)
{
f15 = f15*(x - phi(Omglist)[j]);
}

ideal S = elemSymmId(6);
ideal I = S[1]+4,S[2],S[3],S[4],S[5]-64t,S[6]-64t,f15;
ideal EI = eliminate(I,r(1)*r(2)*r(3)*r(4)*r(5)*r(6));

EI;
factorize(EI[1]);
\end{code}
The code returns the polynomial of degree $15$ and its factorization 
\begin{align*}
f_{15}(x) &= x^{15} - 4480 t x^{12} - 86016 t^2 x^{10} - 20480 t x^{11} + 5836800 t^2 x^9 - 49152 t  x^{10} \\
&\quad + 125829120 t^3 x^7 + 47513600 t^2 x^8 - 5922357248 t^4 x^5 - 2495610880 t^3 x^6 \\
&\quad + 199229440  t^2 x^7 - 93952409600 t^4 x^4 - 23395827712 t^3 x^5 + 671088640 t^2 x^6 \\
&\quad + 1417339207680 t^5 x^2 - 187904819200 t^4 x^3 - 104018739200 t^3 x^4 \\
&\quad + 805306368 t^2 x^5 + 2199023255552 t^6 + 15375982919680 t^5 x \\
&\quad + 1803886264320 t^4 x^2 - 375809638400 t^3 x^3 + 28312424415232 t^5 \\
&\quad - 5819680686080 t^4 x - 2061584302080 t^3 x^2 - 31679678775296 t^4 \\
&\quad - 5497558138880 t^3 x - 4398046511104 t^3 \\
&= (x^5 - 2560 t x^2 - 131072 t^2 - 15360 t x - 16384 t) \\
&\quad (x^{10}  - 1920 t x^7 + 45056 t^2 x^5 - 5120 t x^6 + 921600 t^2 x^4 - 32768 t x^5 \\
&\quad{\ } - 10485760 t^3 x^2 + 4915200 t^2 x^3 - 16777216 t^4 - 115343360 t^3 x \\
&\quad{\ } + 5242880 t^2 x^2 - 213909504 t^3 + 83886080 t^2 x + 268435456 t^2).
\end{align*}
Therefore, the polynomial $f_{15}(x)$ is not a power of an irreducible polynomial over $K$, and $\Gal(L / K) \neq \mathfrak{A}_{6}$.
\end{proof}

\begin{theorem}
Let $C$ be a smooth curve of genus $19$ with a faithful $\mathfrak{A}_{6}$-action. Then $C$ is isomorphic to the Galois closure curve of the morphism $f : \mathbb{P}^1 \rightarrow \mathbb{P}^1$ defined by
\[
f(w) = \dfrac{w^{5}((-783\lambda_{0} + 64)w + 972\lambda_{0})}{972(w - 1)}
\]
where $\lambda_{0}$ is a solution of the quadratic equation $11664\Lambda^{2} -1647\Lambda + 64 = 0$.

The two solutions give rise to isomorphic Galois closure curves. In particular, a smooth $\mathfrak{A}_{6}$-invariant curve of genus $19$ is unique up to isomorphism.
\end{theorem}

\begin{proof}
The first statement follows from Lemma \ref{genus19f} and Lemma \ref{genus19not(1)}. To show that the two value of $\lambda_{0}$ give rise to isomorphic Galois closure curves, we proceed as follows. Let us fix a curve $C$ and an action of $\mathfrak{A}_{6}$ on $C$, and take $\mathfrak{A}_{5}$ and $\mathfrak{A}_{5}'$ to be  non-conjugate icosahedral subgroups. It suffices to show that they correspond to different values of $\lambda_{0}$.

Assume to the contrary that they correspond to the same value of $\lambda_{0}$ and write $f : C / \mathfrak{A}_{5} \rightarrow C / \mathfrak{A}_{6}$ and $f' : C / \mathfrak{A}_{5}' \rightarrow C / \mathfrak{A}_{6}$ for the natural morphism. Then there are two isomorphisms $\alpha : C / \mathfrak{A}_{6} \rightarrow \mathbb{P}^1$ and $\beta : C / \mathfrak{A}_{5} \rightarrow \mathbb{P}^1$ with $f = \alpha^{-1} \circ f_{\lambda_{0}} \circ \beta$ where $f_{\lambda_{0}}$ is rational function define as $f_{\lambda_{0}}(w) = \dfrac{w^{5}((-783\lambda_{0} + 64)w + 972\lambda_{0})}{972(w - 1)}$. Then $\alpha$ sends each $2$ branched points of the ramification index $5$ to $0$ and $\infty$ and the branched point of the ramification index $2$ to $1 \in \mathbb{P}^1$. Similarly, there are two isomorphisms $\alpha' : C / \mathfrak{A}_{6} \rightarrow \mathbb{P}^1$ and $\beta' : C / \mathfrak{A}_{5} \rightarrow \mathbb{P}^1$ with $f' = \alpha'^{-1} \circ f_{\lambda_{0}} \circ \beta'$. Then the projective transformation $\iota := \alpha \circ (\alpha')^{-1}$ on $\mathbb{P}^1$ permutes $0$ and $\infty$ and fixes $1 \in \mathbb{P}^1$. Hence, $\iota(z) = z$ or $\dfrac{1}{z}$. We claim that there exists a projective transformation $\eta$ on $\mathbb{P}^1$ with $f_{\lambda_{0}} = \iota^{-1} \circ f_{\lambda_{0}} \circ \eta$. If $\iota = id_{\mathbb{P}^1}$, then we take $\eta = id_{\mathbb{P}^1}$. On the other hand, if $\iota(z) = \dfrac{1}{z}$, noting that we have $f_{\lambda_{0}}\left(\dfrac{972\lambda_{0}}{(783 \lambda_{0} - 64)w}\right) = \dfrac{1}{f_{\lambda_{0}}(w)}$ by Remark \ref{transform f}, then we take $\eta(w) = \dfrac{972\lambda_{0}}{(783 \lambda_{0} - 64)w}$.
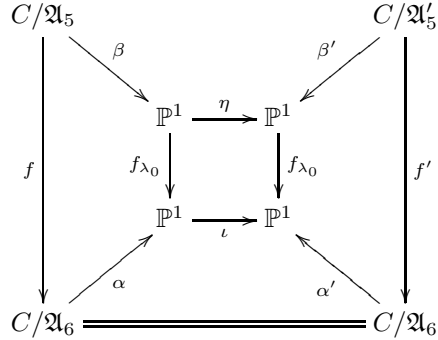
\[
\xymatrix{
C / \mathfrak{A}_{5} \ar[ddd]_-{f} \ar[rd]^{\beta} & & & C / \mathfrak{A}_{5}' \ar[ddd]^-{f'} \ar[ld]_{\beta'} \\
& \mathbb{P}^1 \ar[d]_-{f_{\lambda_{0}}} \ar[r]^-\eta & \mathbb{P}^1 \ar[d]^-{f_{\lambda_{0}}} & \\
& \mathbb{P}^1 \ar[r]_-\iota & \mathbb{P}^1& \\
C / \mathfrak{A}_{6} \ar[ru]_{\alpha} \ar@{=}[rrr] & & & C / \mathfrak{A}_{6} \ar[lu]^{\alpha'}
}
\]
Then we see that $\beta'^{-1} \circ \eta \circ \beta$ is an isomorphism $C / \mathfrak{A}_{5} \rightarrow C / \mathfrak{A}_{5}'$ and $f' \circ (\beta'^{-1} \circ \eta \circ \beta) = f$, i.e., there is an isomorphism $k(C / \mathfrak{A}_{5}') \rightarrow k(C / \mathfrak{A}_{5})$ over $k(C / \mathfrak{A}_{6})$. By Lemma \ref{two isom for conj icosa}(2), the subgroup $\mathfrak{A}_{5}'$ is conjugate to $\mathfrak{A}_{5}$, which is a contradiction. Therefore, the morphism $f$ corresponds to $\lambda_{0}$ or $\overline{\lambda_{0}}$ and the morphism $f'$ corresponds to the other one.
\end{proof}

\end{document}